\documentclass[12 pt]{amsart}
\usepackage{amsfonts,amsmath,amsthm,amssymb} 

\usepackage[OT4]{fontenc}
\usepackage[utf8]{inputenc}
\usepackage{enumerate}
\usepackage{esvect}
\usepackage{bm}

\DeclareMathOperator{\conv}{conv}
\DeclareMathOperator{\card}{card}

\DeclareMathOperator{\Avg}{Avg}

\newcommand{\Vol}[1]{\mathrm{Vol}({#1})}

\renewcommand{\d}[1]{\,\mathrm{d}{#1}}

\theoremstyle{plain}
\newtheorem{theorem}{Theorem}[section]
\newtheorem{corollary}[theorem]{Corollary}
\newtheorem{lemma}[theorem]{Lemma}

\theoremstyle{definition}

\theoremstyle{remark}

\numberwithin{equation}{section}

\begin{document}
\title[Hermite-Hadamard inequalities for simplices]{Refinement of the right-hand side of Hermite-Hadamard inequality for simplices}
\author[M. Nowicka]{Monika Nowicka}
\address{Institute of Mathematics and Physics, UTP University of Science and Technology, al. prof. Kaliskiego 7, 85-796 Bydgoszcz, Poland}
\email{monika.nowicka@utp.edu.pl}

\author[A. Witkowski]{Alfred Witkowski}
\email{alfred.witkowski@utp.edu.pl}
\subjclass[2010]{26D15}
\keywords{Hermite-Hadamard inequality, convex function, simplex, barycentric coordinates}
\date{3.09.2015}

\begin{abstract}
We establish a new refinement of the right-hand side of the Hermite-Hadamard inequality for simplices, based on the average values of a convex function over the faces of  a simplex and over the values at their barycenters.
\end{abstract}
\maketitle
\section{Introduction}
The classical Hermite-Hadamard inequality \cite{Had} states that for a convex function $f\colon[a,b]\to\mathbb{R}$
\begin{equation}
f\left(\frac{a+b}{2}\right)\leq \frac{1}{b-a}\int_a^b f(t)\d{t}\leq \frac{f(a)+f(b)}{2}.
\label{eq:HH-dim1}
\end{equation}
There are many generalizations. One of them (\cite{Bes}) says that if $\Delta\subset\mathbb{R}^n$ is a simplex with barycenter $\mathbf{b}$ and vertices $\mathbf{x}_0,\dots,\mathbf{x}_n$ and $f\colon\Delta\to\mathbb{R}$ is convex, then
\begin{equation}\label{eq:HH-Bess}
	f(\mathbf{b})\leq \frac{1}{\Vol{\Delta}}\int_\Delta f(\mathbf{x})\d{\mathbf{x}}\leq \frac{f(\mathbf{x}_0)+\dots f(\mathbf{x}_n)}{n+1}.
\end{equation}
In this paper we aim to improve the right-hand side of the last inequality. 

\section{Definitions and lemmas}
For a fixed  natural number $n\geq 1$ let $N=\{0,1,\dots,n\}$. Suppose $\mathbf{x}_0,\dots, \mathbf{x}_n\in\mathbb{R}^n$ are such that the vectors $\vv{\mathbf{x}_0\mathbf{x}_i},\ i=1,\dots,n$ are linearly independent.\\
Suppose  $K$ is a nonempty subset of $N$ with $k+1$ elements ($0\leq k\leq n$) and denote the elements of $\{\mathbf{x}_i\colon i\in K\}$ by $\{\mathbf{y}_0,\dots,\mathbf{y}_k\}$. The set $\Delta_K=\conv \{\mathbf{y}_0,\dots,\mathbf{y}_k\}$ is called a \textit{simplex} (or a $k$-\textit{simplex} if we want to emphasize its dimension). 
The $n$-simplex $\Delta_N$ will be denoted by $\Delta$.

Given two sets $K\subset L\subset N$ the simplex $\Delta_K$ is called a \textit{face} (or $k$-\textit{face}) of $\Delta_L$.

The points in $\Delta_K$ admit  unique representation of the form
$$\mathbf{y}=\sum_{i=0}^k \alpha_i \mathbf{y}_i,\quad \alpha_i\geq 0,\quad \sum_{i=0}^k \alpha_i=1.$$
The $k+1$-tuple $(\alpha_0,\dots,\alpha_k)$ is called \textit{barycentric coordinates}. The point 
\begin{equation*}
\mathbf{b}_K=\frac{1}{k+1}(\mathbf{y}_0+\dots+\mathbf{y}_k)
\end{equation*}
is called the \textit{barycenter} of $\Delta_K$.

If $\Sigma$ is a $k$-dimensional simplex and $f:\Sigma\to\mathbb{R}$ is an integrable function, the we shall denote its average value over the simplex by $\Avg(f,\Sigma)$, i.e.
$$\Avg(f,\Sigma)=\frac{1}{\Vol{\Sigma}}\int_\Sigma f(\mathbf{x})\d{\mathbf{x}}.$$

The integration here is with respect to the $k$-dimensional Lebesgue measure denoted by $\d{\mathbf{x}}$ and  $\Vol{\Sigma}$ denotes the $k$-dimensional volume. 

With this notation we can write the right-hand side of \eqref{eq:HH-Bess} as
$$(n+1)\Avg(f,\Delta)\leq \sum_{i=0}^n \Avg(f,\Delta_{\{i\}}).$$

The set $E_k=\{(\alpha_1,\dots,\alpha_k)\colon \alpha_i\geq 0, \sum_{i=1}^k\alpha_i\leq 1\}\subset\mathbb{R}^k$ is called a \textit{standard simplex}.

For every $k$-simplex $\Delta_K$ we define a one-to-one mapping $\varphi_K\colon E_k\to\Delta_K$ given by the formula
\begin{equation*}
\varphi_K(\alpha_1,\dots,\alpha_k)=\mathbf{y}_0+\sum_{i=1}^k\alpha_i(\mathbf{y}_i-\mathbf{y}_0).
\end{equation*}
Note that $\partial\varphi_K/\partial\alpha_i=\mathbf{y}_i-\mathbf{y}_0$, so the absolute value of its Jacobian equals $|\det D\varphi_K|=k!\Vol{\Delta_K}$. This means, that for an integrable function $g\colon \Delta_K\to\mathbb{R} $ the identity
\begin{equation}
\frac{1}{\Vol{\Delta_K}}\int_{\Delta_K} g(\mathbf{x})\d{\mathbf{x}}=k!\int_{E_k} g(\varphi_K(\boldsymbol{\alpha}))\d{\boldsymbol{\alpha}}
\label{eq:integral}
\end{equation}
holds. In particular $\Vol{E_k}=1/k!$.

A \textit{partition} of $N$ is a set 
$$\mathcal{K}=\{K_i:i=1\dots p,\ K_i\subset N,\ \bigcup_{i=1}^p K_i=N,\ K_i\cap K_j=\emptyset\}$$
We say that $\mathcal{L}$ \textit{refines} $\mathcal{K}$ (and write $\mathcal{L}\prec\mathcal{K}$) if every element of $\mathcal{L}$ is a subset of an element of $\mathcal{K}$.

\section{Faces based refinemets}

The right-hand side of \eqref{eq:HH-Bess} shows that  the average of a convex function over a simplex
can be bounded by the convex combination of its average values over its $0$-faces. Our main result - Theorem \ref{thm:main} and its corollaries show how to generalize this fact.
\begin{theorem}\label{thm:main}
	Let $\mathcal{K}=\{K_1,\dots,K_p\}$ be a partition of $N$  and $f\colon \Delta\to\mathbb{R}$ be a convex function. Define
	$$F(\mathcal{K})=\sum_{i=1}^p \card K_i\Avg(f,\Delta_{K_i}).$$
If $\mathcal{L}\prec \mathcal{K}$, then
\begin{equation}
F(\mathcal{K})\leq F(\mathcal{L})\,.
\label{eq:thmmain2}
\end{equation}
\end{theorem}

To prove this theorem we shall need Lemma \ref{lem:main}.
\begin{lemma}\label{lem:main}
	Let $K,L\subset N$ be two disjoint, nonempty sets with $\card K=k+1$ and $\card L=l+1$. Further, let $f\colon\Delta\to\mathbb{R}$ be a convex function. Then
$$	\Avg(f,\Delta_{K\cup L})\leq \frac{k+1}{k+l+2}\cdot \Avg(f,\Delta_{K}) + 
\frac{l+1}{k+l+2}\cdot  \Avg(f,\Delta_{L}).$$
	
\end{lemma}
\begin{proof}
	Denote by $\mathbf{u}_0,\dots,\mathbf{u}_k$ the vertices of $\Delta_K$ and by $(\alpha_0,\dots,\alpha_k)$ its barycentric coordinates. Similarly, let $\mathbf{v}_0,\dots,\mathbf{v}_l$ be the vertices of $\Delta_L$ and by $(\beta_0,\dots,\beta_l)$ its barycentric coordinates. Every point $\mathbf{x}\in\Delta_{K\cup L}$ can be represented as
\begin{equation*}
\mathbf{x}=\sum_{i=0}^k \alpha_i'\mathbf{u}_i+\sum_{j=0}^l \beta_j'\mathbf{v}_j,\ \alpha_i',\beta_j'\geq 0,\ \sum_{i=0}^k \alpha_i'+\sum_{j=0}^l \beta_j'=1. 
\end{equation*}	
Let $\sum_{i=0}^k \alpha_i'=s,\ \sum_{j=0}^l \beta_j'	=1-s$ and $\alpha_i=\alpha_i'/s,\ \beta_j=\beta_j'/(1-s)$	(if division by zero occurs we assume the result is zero). Then
\begin{equation*}
\mathbf{x}=s\sum_{i=0}^k \alpha_i \mathbf{u}_i+(1-s)\sum_{j=0}^l \beta_j \mathbf{v}_j.
\end{equation*}	
Let $\Phi\colon [0,1]\times E_k\times E_l\to \Delta_{K\cup L}$ be defined by
\begin{align*}
	\Phi(s,&\alpha_1,\dots,\alpha_k,\beta_1,\dots,\beta_l)	\\
	&	=s\left(\mathbf{u}_0+\sum_{i=0}^k \alpha_i (\mathbf{u}_i-\mathbf{u}_0)\right)+(1-s)\left(\mathbf{v}_0+\sum_{j=0}^l \beta_j (\mathbf{v}_j-\mathbf{v}_0)\right)\\
	&=s\varphi_K(\alpha_1,\dots,\alpha_k)+(1-s)\varphi_L(\beta_1,\dots,\beta_l).
\end{align*}
Its Jacobian equals
\begin{align}
	\det D\Phi&=\begin{vmatrix}
		\Phi_s'\\\vdots\\\Phi_{\alpha_i}'\\ \vdots\\\Phi_{\beta_j}'\\\vdots
	\end{vmatrix}
		=\begin{vmatrix}
			\mathbf{u}_0-\mathbf{v}_0 +\sum_{i=0}^k \alpha_i (\mathbf{u}_i-\mathbf{u}_0)-\sum_{j=0}^l \beta_j (\mathbf{v}_j-\mathbf{v}_0)\\\vdots\\s(\mathbf{u}_i-\mathbf{u}_0)\\ \vdots\\(1-s)(\mathbf{v}_j-\mathbf{v}_0)\\\vdots
		\end{vmatrix}\notag\\
	&=s^k(1-s)^l \begin{vmatrix}
	\mathbf{u}_0-\mathbf{v}_0 \\\vdots\\\mathbf{u}_i-\mathbf{u}_0\\ \vdots\\\mathbf{v}_j-\mathbf{v}_0	\\\vdots
	\end{vmatrix}=
	s^k(1-s)^l \begin{vmatrix}
	\mathbf{u}_0-\mathbf{v}_0 \\\vdots\\\mathbf{u}_i-\mathbf{v}_0\\ \vdots\\\mathbf{v}_j-\mathbf{v}_0	\\\vdots
	\end{vmatrix}\notag\\
	&=\pm(k+l+1)!s^k(1-s)^l\Vol{\Delta_{K\cup L}}.\label{Jacobian Phi}
	\end{align}

Changing variable in the integral yields
\begin{equation}
\int_{\Delta_{K\cup L}}f(\mathbf{x})d\mathbf{x}	=\int_0^1ds\int_{E_k}d\boldsymbol{\alpha}\int_{E_l}d\boldsymbol{\beta}f(\Phi(s,\boldsymbol{\alpha},\boldsymbol{\beta}))|\det D\Phi|.	
\label{eq:changing variable}
\end{equation}
Using convexity of $f$  and the formula  \eqref{eq:integral} we obtain
\begin{align}
\int_0^1ds&\int_{E_k}d\boldsymbol{\alpha}\int_{E_l}d\boldsymbol{\beta}f(\Phi(s,\boldsymbol{\alpha},\boldsymbol{\beta}))s^k(1-s)^l\notag\\
	&=\int_0^1s^k(1-s)^lds\int_{E_k}d\boldsymbol{\alpha}\int_{E_l}d\boldsymbol{\beta}f(s\varphi_K(\boldsymbol{\alpha})+(1-s)\varphi_L(\boldsymbol{\beta})	)\notag\\
	&\leq \phantom{+}\int_0^1s^{k+1}(1-s)^lds\int_{E_k}f(\varphi_K(\boldsymbol{\alpha}))d\boldsymbol{\alpha}\int_{E_l}d\boldsymbol{\beta}\notag\\
	&\phantom{\leq}+\int_0^1s^k(1-s)^{l+1}ds\int_{E_k}d\boldsymbol{\alpha}\int_{E_l}f(\varphi_L(\boldsymbol{\beta})	)d\boldsymbol{\beta}\notag\\
	&=\frac{(k+1)!l!}{(k+l+2)!}\cdot
	\frac{\Avg(f,\Delta_K)}{k!}\cdot\frac{1}{l!}
	+\frac{k!(l+1)!}{(k+l+2)!}\cdot\frac{1}{k!}\cdot\frac{\Avg(f,\Delta_L)}{l!}.\label{eq:calculating integral}
\end{align} 
Lemma follows from \eqref{Jacobian Phi}, \eqref{eq:changing variable} and \eqref{eq:calculating integral}.
\end{proof}
Now we can prove our main result.
\begin{proof}[Proof of Theorem \ref{thm:main}]

It follows from Lemma \ref{lem:main} that the function
$$K\mapsto  \card K\Avg(f,\Delta_K)$$
is subadditive on disjoint sets, so \eqref{eq:thmmain2} follows by mathematical induction on the cardinality of partition. 
\end{proof}
Since $\{\{0\},\dots,\{n\}\}\prec\mathcal{K}\prec\{N\}$ for any partition $\mathcal{K}$ we obtain the refinements of the Hermite-Hadamard inequality.
\begin{corollary}\label{coro:main}
For any partition $\mathcal{K}$ the inequalities
$$(n+1)\Avg(f,\Delta)\leq F(\mathcal{K})\leq f(\mathbf{x}_0)+\dots+f(\mathbf{x}_n)$$
hold.
\end{corollary}

Applying Corollary \ref{coro:main} to all possible pairs consisting of a vertex and its opposite face one obtains the following result.
\begin{corollary}\label{cor:barycenter}
Let $f\colon \Delta\to\mathbb{R}$ be a convex function.
Then the inequality
\begin{align*}
\Avg(f,\Delta)&
\leq\frac{1}{n+1} \frac{f(\mathbf{x}_0)+\dots+f(\mathbf{x}_n)}{n+1}\\
&+\frac{n}{n+1}\cdot\frac{1}{n+1}\sum_{\substack{{K\subset N}\\{\card K=n}}}\Avg(f,\Delta_K)
\end{align*}
holds.
\end{corollary}

If $N$ can be divided into disjoint subsets of the same cardinality, then applying Corollary \ref{coro:main} to all possible partitions and summing the obtained inequalities one gets the following corollary.

\begin{corollary}\label{coro:main1}
Let $f\colon \Delta\to\mathbb{R}$ be a convex function and $d$ be a divisor of $n+1=\card N$.
Then
$$\Avg(f,\Delta)\leq
\frac{1}{\binom {n+1} {d}}\sum_{\substack{{K\subset N}\\{\card K=d}}}\Avg(f,\Delta_{K})\,.$$
\end{corollary}

It is known (see e.g. \cite[p. 125]{Som}) that the volume of a regular simplex $\Delta\subset\mathbb{R}^n$ with unit edges equals $\frac{1}{n!}\sqrt{\frac{n+1}{2^n}}$. In this case  Corollary \ref{coro:main1} yields

\begin{corollary}
Let $f$ be a convex function defined on a regular simplex $\Delta$ with unit edges.
Let $d$ be a divisor of $n+1=\card N$.
Then
\begin{align*}
&\int_{\Delta}f(\mathbf{x})\d{\mathbf{x}}\\
&\leq \left[\binom {n}{d-1}\right]^{-2}\frac{1}{(n+1-d)!}
 \sqrt{\frac{d}{n+1}\frac{1}{2^{n+1-d}}}\sum_{\substack{{K\subset N}\\{\card K=d}}}\int_{\Delta_{K}}f(\mathbf{x})\d{\mathbf{x}}.
\end{align*}
\end{corollary}
\section{Barycentric refinements}
In case of one dimension the right-hand side of \eqref{eq:HH-dim1} can be refined as follows (\cite{Ham,WW}):
\begin{equation}
\frac{1}{b-a}\int_a^b f(t)\d{t}\leq \frac{1}{2}\left[f\left(\frac{a+b}{2}\right)+\frac{f(a)+f(b)}{2}\right].
\label{eq:HH-lepsza}
\end{equation}
Given the fact that both $f(a)$ and $f(b)$ are the barycenters of $0$-faces, we see that in this case the average value of $f$ over a simplex is bounded by a convex combination of its values on barycenters of all faces. We shall generalize this result to simplices.

Let us begin with lemma.
\begin{lemma}\label{lem:barycenter}
	Let $\mathbf{b}$ be the barycenter of $\Delta$ and let $K_i=N\setminus\{i\}$. If $f:\Delta\to\mathbb{R}$ is convex, then
	$$\Avg(f,\Delta)\leq
	\frac{1}{n+1}f(\mathbf{b})+\frac{n}{n+1}\cdot\frac{1}{n+1}\displaystyle \sum_{i=0}^n \Avg(f,\Delta_{K_i}).$$
\end{lemma}
\begin{proof}
Note that  we can identify the partitions of $N$ with partitions of the set of vertices of a simplex. 
	The barycenter divides $\Delta$ into $n+1$ simplices $$B_i=\conv\{\mathbf{x}_0,\dots,\mathbf{x}_{i-1},\mathbf{b},\mathbf{x}_{i+1},\dots,\mathbf{x}_n\},\quad i=0,\dots,n.$$
	For each of them we split its vertices into two groups: $\{\mathbf{b}\}$ and $\{\mathbf{x}_j: j\neq i\}$ and apply Lemma \ref{lem:main} to these partitions. Taking into account that $\Vol{B_i}=\frac{1}{n+1}\Vol{\Delta}$ for every $i$, we get
	
	\begin{align}
		\frac{\int_{{B_i}}f(\mathbf{x})\d{\mathbf{x}}}{\Vol{\Delta}}&=
	\frac{1}{n+1}\Avg(f,B_i)	\notag\\
		&	\leq \frac{1}{n+1}\left[\frac{1}{n+1}f(\mathbf{b})+\frac{n}{n+1}\Avg(f,\Delta_{K_i})\right].\label{neq:xyz1}
	\end{align}

	Summing these inequalities completes the proof.
\end{proof}
Applying \eqref{eq:HH-Bess} to the right-hand side of \eqref{neq:xyz1} we  generalize the inequality \eqref{eq:HH-lepsza}.
\begin{theorem}\label{thm:th2}
Let $\Delta\subset\mathbb{R}^n$ be an arbitrary simplex with vertices $\mathbf{x}_0,\dots,\mathbf{x}_n$ and let $\mathbf{b}$ be the barycenter of $\Delta$.
Then for every convex function $f:\Delta\to\mathbb{R}$ we have the estimate
$$\Avg(f,\Delta)\leq\frac{1}{n+1}f(\mathbf{b})+\frac{n}{n+1}\frac{f(\mathbf{x}_0)+\dots+f(\mathbf{x}_n)}{n+1}.$$
\end{theorem}
But we can do much better: instead of \eqref{eq:HH-Bess} we can use inductively Lemma \ref{lem:barycenter} to $(n-1)$-faces etc. and continue this process until we reach $0$-faces. Thus we obtain 
\begin{theorem}
$$\Avg(f,\Delta)\leq
\frac{1}{n+1}\sum_{k=1}^{n+1} \frac{1}{\binom {n+1} {k}}\sum_{\substack{{K\subset N}\\{\card K=k}}}f(\mathbf{b}_K)$$
\end{theorem}

Combining the results of Corollary \ref{coro:main} and Theorem \ref{thm:th2} one can produce various new upper bounds for the average value of $f$ over the simplex. Below we show one of them.


Splitting the vertices of $\Delta$ into the maximum number of groups of $k$ elements and one group of $l$ elements, where $l<k$, and taking the average over all such splits one gets.
\begin{corollary}
Let $\Delta\subset\mathbb{R}^n$ be an arbitrary simplex and let $\mathbf{x}_0,\dots,\mathbf{x}_n$ be its vertices.
$$\Avg(f,\Delta)\leq
\alpha_n \frac{1}{\binom {n+1} {k}}\sum_{\substack{{K\subset N}\\{\card K=k}}}f\left(\mathbf{b}_K\right)+(1-\alpha_n)\frac{f(\mathbf{x}_0)+\dots+f(\mathbf{x}_n)}{n+1},$$
where $\alpha_n=\displaystyle\frac{\left\lfloor \frac{n+1}{k}\right\rfloor}{n+1}$.
\end{corollary}

In particular we obtain.
\begin{corollary}
Let $\Delta\subset\mathbb{R}^n$ be an arbitrary simplex with vertices $\mathbf{x}_0,\dots,\mathbf{x}_n$.
Then for every convex function $f:\Delta\to\mathbb{R}$ we have
$$\Avg(f,\Delta)\leq
\alpha_n \frac{1}{\binom {n+1} {2}}\sum_{\substack{0 \leq i<j \leq n}}f\left(\frac{\mathbf{x}_i+\mathbf{x}_j}{2}\right)+\beta_n\frac{f(\mathbf{x}_0)+\dots+f(\mathbf{x}_n)}{n+1},$$
where $\alpha_n=\displaystyle\frac{\left\lfloor \frac{n+1}{2}\right\rfloor}{n+1}$ and $\beta_n=\displaystyle\frac{ \left\lceil \frac{n+1}{2}\right\rceil}{n+1}$.
\end{corollary}

\end{document}